\documentclass[11pt,reqno]{amsart}
\usepackage[utf8]{inputenc}

\usepackage[margin=1in]{geometry}
\usepackage{amsmath}
\usepackage{amsfonts}
\usepackage{amssymb}
\usepackage{amsthm}
\usepackage{enumitem}
\usepackage{mathrsfs}
\usepackage{mathtools}
\usepackage{tikz-cd}
\usepackage{enumitem}

\usepackage[colorlinks=true,hyperindex, linkcolor=magenta, pagebackref=false, citecolor=cyan,pdfpagelabels]{hyperref}

\newtheorem{theorem}{Theorem}[section]

\newtheorem{lemma}[theorem]{Lemma}

\theoremstyle{definition}

\newtheorem{remark}[theorem]{Remark}

\newcommand{\vp}{\varphi}

\newcommand{\bG}{{\mathbf G}}
\newcommand{\bH}{{\mathbf H}}

\newcommand{\bQ}{{\mathbf Q}}
\newcommand{\bR}{{\mathbf R}}

\newcommand{\bZ}{{\mathbf Z}}

\newcommand{\sA}{{\mathscr A}}
\newcommand{\sB}{{\mathscr B}}

\newcommand{\sG}{{\mathscr G}}
\newcommand{\sH}{{\mathscr H}}

\newcommand{\sO}{{\mathscr O}}

\newcommand{\sT}{{\mathscr T}}

\newcommand{\sX}{{\mathscr X}}

\DeclareMathOperator{\Ad}{Ad}

\DeclareMathOperator{\GL}{GL}

\DeclareMathOperator{\Hom}{Hom}

\DeclareMathOperator{\id}{id}

\DeclareMathOperator{\Spec}{Spec}

\newcommand{\wtil}[1]{\wtil}

\title{Morphisms of character varieties}
\author{Sean Cotner}

\begin{document}
\bibliographystyle{halpha-abbrv}
\maketitle

\begin{abstract}
Let $k$ be a field, let $H \subset G$ be (possibly disconnected) reductive groups over $k$, and let $\Gamma$ be a finitely generated group. Vinberg and Martin have shown that the induced morphism of character varieties
\[
\underline{\Hom}_{k\textrm{-gp}}(\Gamma, H)/\!/H \to \underline{\Hom}_{k\textrm{-gp}}(\Gamma, G)/\!/G
\]
is finite. In this note, we generalize this result (with a significantly different proof) by replacing $k$ with an arbitrary locally noetherian scheme, answering a question of Dat. Along the way, we use Bruhat--Tits theory to establish a few apparently new results about integral models of reductive groups over discrete valuation rings.
\end{abstract}

\section{Introduction}

Let $S$ be a locally noetherian scheme, and let $H$ and $G$ be smooth $S$-affine $S$-group schemes with reductive fibers and finite \'etale component groups. Let $f: H \to G$ be an $S$-homomorphism and let $h: \Gamma' \to \Gamma$ be a homomorphism of finitely generated groups. If $\underline{\Hom}_{S\textrm{-gp}}(\Gamma, G)$ denotes the scheme of homomorphisms from $\Gamma$ to $G$, then there is a natural $S$-morphism
\[
F: \underline{\Hom}_{S\textrm{-gp}}(\Gamma, H)/\!/H \to \underline{\Hom}_{S\textrm{-gp}}(\Gamma', G)/\!/G
\]
between the GIT quotients. The goal of this note is to prove the following theorem, answering a question of J.-F.\ Dat \cite[Conjecture 5.16]{Dat-IHES}.

\begin{theorem}\label{theorem:main-theorem}
If $f$ is a finite morphism and $h(\Gamma')$ is of finite index in $\Gamma$, then $F$ is finite.
\end{theorem}

Theorem~\ref{theorem:main-theorem} has been proved in various special cases in the literature.
\begin{enumerate}
    \item If $S$ is the spectrum of a field, $\Gamma = \Gamma'$, and $h = \id_{\Gamma}$ then Theorem~\ref{theorem:main-theorem} is due to Vinberg \cite{Vinberg} in characteristic 0 and Martin \cite{martin03} in positive characteristic.
    \item If $\Gamma = \Gamma' = \bZ$, $h = \id_{\bZ}$, and $H$ and $G$ are reductive, then Theorem~\ref{theorem:main-theorem} follows from the Chevalley--Steinberg theorem (for more general $H$ and $G$, see \cite[Lemma 2.2]{dhkm_finiteness}).
    \item If $S = \Spec W(k)$ for an algebraically closed field $k$ of characteristic $p > 0$, $\Gamma$ is a finite group of order prime to $p$, and $H$ and $G$ are semisimple, then \cite[Appendix A]{Griess-Ryba} shows that the source and target of $F$ are finite \'etale, and in particular Theorem~\ref{theorem:main-theorem} holds.
    \item If $p$ is a prime number, $S = \Spec \overline{\bZ}[\frac{1}{p}]$, the function $h$ is injective, and $\Gamma = W_F^0/P$ for a finite extension $F/\bQ_p$ (where $W_F^0$ is the ``discretized Weil group" introduced in \cite{dhkm_moduli} and $P$ is an open subgroup of wild inertia which is normal in $W_F$), then Theorem~\ref{theorem:main-theorem} is proved with some assumptions on $G$ in \cite[Corollary 2.4, Corollary 2.5]{dhkm_finiteness}.
\end{enumerate}

\begin{remark}\label{remark:warning}
    The statement of Theorem~\ref{theorem:main-theorem} is slightly strange: when $\Gamma = \Gamma'$ is a free group on $N$ letters, we have $\underline{\Hom}_{S\textrm{-gp}}(\Gamma, H) \cong H^N$ and the induced morphism of quotient stacks
    \[
    [H^N/H] \to [G^N/G]
    \]
    is almost never universally closed. Indeed, if $f$ is monic then pulling back via the map $S \to [G^N/G]$ corresponding to $(1, \dots, 1)$ gives $G/H \to S$, which is universally closed if and only if $G/H$ is finite.
\end{remark}

We will prove Theorem~\ref{theorem:main-theorem} by verifying the valuative criterion of properness. Let us briefly illustrate the strategy, ignoring a few subtleties; we remark that it is quite different from the strategies of \cite{Vinberg} and \cite{martin03} even over a field. Suppose given a commutative diagram
\[
\begin{tikzcd}
    \Spec(K) \arrow[r, "y"] \arrow[d]
    &\underline{\Hom}_{S\textrm{-gp}}(\Gamma, H)/\!/H \arrow[d, "F"] \\
    \Spec(A) \arrow[r, "x"]
    &\underline{\Hom}_{S\textrm{-gp}}(\Gamma', G)/\!/G
\end{tikzcd}
\]
in which $A$ is a DVR with fraction field $K$. After passing to a local extension of $A$, which we may assume to be complete, Lemma~\ref{lemma:git-quotient-points} shows that $x$ arises from a homomorphism $\vp: \Gamma' \to G(A)$ and $y$ arises from a homomorphism $\psi: \Gamma \to H(K)$ with closed $H_K$-orbit such that $\vp$ and $f \circ \psi \circ h$ are $G(K)$-conjugate. We must show that after extending $A$ further, $\psi(\Gamma)$ is $H(K)$-conjugate to a subgroup of $H(A)$. Using the $G(K)$-conjugacy of $f \circ \psi \circ h$ and $\vp$, one shows that $\psi(\Gamma)$ is a bounded subgroup of $H(K)$, and using the closedness of the orbit of $\psi$ it follows from \cite{Richardson} that the Zariski closure $H_1$ of $\psi(\Gamma)$ in $H_K$ is reductive (but possibly disconnected). At this point, the key input is the following theorem (applied to $B = \psi(\Gamma)$), which collects the results of Sections~\ref{section:extending-groups} and \ref{section:extending-homs}.

\begin{theorem}\label{theorem:main-theorem-2}
    Let $A$ be an excellent henselian DVR with fraction field $K$.
    \begin{enumerate}
        \item\label{item:main-2-1} (Lemmas~\ref{lemma:extending-reductive-groups}, \ref{lemma:extending-reductive-groups-2}) Let $H_1$ be a reductive $K$-group, and let $B \subset H_1(K)$ be a bounded subgroup. There exists a finite local extension of DVRs $A \subset A'$ with fraction field $K'$ and a smooth affine model $\sH_1$ of $(H_1)_{K'}$ over $A'$ such that $\sH_1^0$ is a reductive group scheme, $\sH_1/\sH_1^0$ is finite \'etale, and $B \subset \sH_1(A')$.
        \item\label{item:main-2-2} (Lemma~\ref{lemma:griess-ryba}) Let $\sH_1$ and $H$ be smooth affine $A$-group schemes such that $H^0$ is reductive and $H/H^0$ is finite \'etale, and let $f_1: (\sH_1)_K \to H_K$ be a $K$-homomorphism. There exists a finite local extension of DVRs $A \subset A'$ with fraction field $K'$ and $h \in H^0(K')$ such that $\Ad(h) \circ (f_1)_{K'}$ extends to an $A'$-homomorphism $f: (\sH_1)_{A'} \to H_{A'}$.
    \end{enumerate}
\end{theorem}

In the case that $S = \Spec W(k)$ for an algebraically closed field $k$ of characteristic $p > 0$, $\sH_1$ is constant of order prime to $p$, and $H$ is semisimple (with connected fibers), Theorem~\ref{theorem:main-theorem-2}(\ref{item:main-2-2}) is proved in \cite[Lemma A.8]{Griess-Ryba}.

The proof of Theorem~\ref{theorem:main-theorem-2}(\ref{item:main-2-1}) proceeds by first constructing (after a finite local extension of $A$) a \textit{single} integral model of $H_1$ with good properties. The group of $A$-points of any such model can be realized as the stabilizer of a hyperspecial point in the Bruhat--Tits building $\sB(H_1^0)$. Next, we show that for any bounded subgroup $B$ of $H_1(K)$, there is a finite extension $K'/K$ such that $B$ stabilizes a hyperspecial point of $\sB((H_1^0)_{K'})$, so we obtain (\ref{item:main-2-1}) from the conjugacy of hyperspecial points. Item (\ref{item:main-2-2}) follows from (\ref{item:main-2-1}) applied to $B = f_1(H(A))$, the conjugacy of hyperspecial points, and a simple lemma (Lemma~\ref{lemma:kaletha-prasad}).

\subsection{Acknowledgements}

I would like to thank Ben Church for several long and interesting conversations which helped to crystallize the ideas of this paper. I thank Brian Conrad and Ravi Vakil for helpful conversations which led to corrections and simplifications. I thank Jarod Alper, Vytautas Pa\v{s}k\={u}nas, and Julian Quast for asking a useful question, and I thank the anonymous referee for suggesting a simplification of the proof of Lemma~\ref{lemma:extending-reductive-groups}.

\section{Preliminaries}\label{section:integral-git}

We recall first the \textit{existence part of the (noetherian) valuative criterion of properness} for morphisms of stacks. For simplicity, let $f: \sX \to Y$ be a finite type morphism of algebraic stacks such that $Y$ is a locally noetherian scheme. We say that $f$ satisfies the existence part of the valuative criterion of properness if, for every DVR $A$ with fraction field $K$ and every solid commutative diagram
\[
\begin{tikzcd}
    \Spec K' \arrow[r, dotted] \arrow[d, dotted]
    &\Spec K \arrow[r] \arrow[d]
    &\sX \arrow[d] \\
    \Spec A' \arrow[r, dotted] \arrow[rru, dotted]
    &\Spec A \arrow[r]
    &Y
\end{tikzcd}
\]
there exists a local extension $A \subset A'$ of DVRs with fraction field $K'$ such that the dotted diagram can be filled in to become commutative. (Note that this differs slightly from the definition in \cite[\href{https://stacks.math.columbia.edu/tag/0CLA}{Tag 0CLA}]{stacks-project}, in which $A$ and $A'$ are only required to be valuation rings; we have also simplified the definition by assuming that $Y$ has no stabilizers.)

The relevance of this definition for our purposes is that if $f$ is finite type and quasi-separated, then \cite[\href{https://stacks.math.columbia.edu/tag/0H2C}{Tag 0H2C}, \href{https://stacks.math.columbia.edu/tag/0CLX}{Tag 0CLX}]{stacks-project} shows that $f$ satisfies the existence part of the valuative criterion of properness if and only if $f$ is universally closed. (Strictly speaking, \cite[\href{https://stacks.math.columbia.edu/tag/0CLX}{Tag 0CLX}]{stacks-project} is only a statement about the \textit{non-noetherian} valuative criterion, but its proof works mutatis mutandis for the noetherian valuative criterion by using \cite[\href{https://stacks.math.columbia.edu/tag/0H2B}{Tag 0H2B}]{stacks-project} in place of \cite[\href{https://stacks.math.columbia.edu/tag/0CL2}{Tag 0CL2}]{stacks-project}.)

We recall also the notion of geometric reductivity from \cite[Definition 9.1.1]{Alper-adequate}. The precise definition is of no relevance to us, but we recall the following two facts which we will use without comment in what follows.
\begin{enumerate}
    \item \cite[Theorem 9.7.6]{Alper-adequate} If $S$ is a scheme, then a smooth affine\footnote{The published version of \cite[Theorem 9.7.6]{Alper-adequate} omits the word ``affine" in this statement. Without this assumption, the statement is false: indeed, abelian schemes are geometrically reductive. In the 2018 arXiv version of \cite[Theorem 9.7.6]{Alper-adequate}, this is corrected, but there is an additional assumption that $G/G^0$ is separated. This latter assumption is superfluous: all group schemes over a field are separated, so if $G$ is geometrically reductive smooth affine then the theorem shows $G$ has reductive fibers, and thus \cite[Proposition 3.1.3]{Conrad} shows that $G/G^0$ is a separated \'etale $S$-group of finite presentation. I thank Jarod Alper, Vytautas Pa\v{s}k\={u}nas, and Julian Quast for inquiring about this discrepancy.} $S$-group scheme $G$ is geometrically reductive if and only if the relative identity component $G^0$ is a reductive group scheme and the quotient $G/G^0$ is finite \'etale over $S$.
    \item \cite[Theorems 5.3.1, 6.3.3, 9.1.4]{Alper-adequate} If $S$ is a locally noetherian scheme, $G$ is a geometrically reductive $S$-group scheme, and $p: X \to S$ is a finite type $S$-affine $S$-scheme equipped with a $G$-action, then the GIT quotient $X/\!/G \coloneqq \mathrm{Spec}_S(p_*\sO_X)^G$ is of finite type over $S$ and the natural map $[X/G] \to X/\!/G$ from the quotient stack is surjective and universally closed.
\end{enumerate}

\begin{lemma}\label{lemma:adequate-valuative-criterion}
    Let $S$ be a locally noetherian scheme, let $X$ be a finite type $S$-affine $S$-scheme, and let $G$ be a geometrically reductive $S$-group scheme acting on $X$. The natural map $\pi: [X/G] \to X/\!/G$ satisfies the existence part of the valuative criterion of properness.
\end{lemma}

\begin{proof}
    As noted above, $\pi$ is universally closed. Since $G \to S$ is quasi-compact, the quotient stack $[X/G]$ is evidently quasi-separated and of finite type over $S$. By cancellation, it follows that $\pi$ is quasi-separated and of finite type, and thus the result follows from \cite[\href{https://stacks.math.columbia.edu/tag/0CLX}{Tag 0CLX}]{stacks-project} (suitably adapted to the noetherian setting, as indicated above).
\end{proof}

\begin{lemma}\label{lemma:git-quotient-points}
    Let $A$ be a DVR, let $G$ be a geometrically reductive smooth affine $A$-group scheme, and let $X$ be an affine $A$-scheme on which $G$ acts. If $x \in (X/\!/G)(A)$, then there is a local extension of DVRs $A \subset A'$ such that $x_{A'}$ lies in the image of $X(A') \to (X/\!/G)(A')$.
\end{lemma}

\begin{proof}
    By Lemma~\ref{lemma:adequate-valuative-criterion} and surjectivity, we may pass to a local extension of $A$ to assume that there exists $x_0 \in [X/G](A)$ mapping to $x$. By definition, $X \to [X/G]$ is an \'etale $G$-torsor, so we may pass to a further local extension of $A$ to assume that $x_0$ lifts to $x_1 \in X(A)$.
\end{proof}

\section{Extending reductive groups over DVRs}\label{section:extending-groups}

The goal of this section is to show that, after passing to a quasi-finite local extension of $A$, any (possibly disconnected) reductive $K$-group $G$ admits a geometrically reductive smooth affine model over $A$. This result (Lemma~\ref{lemma:extending-reductive-groups}) will be complemented by Lemma~\ref{lemma:extending-reductive-groups-2} which will show that, \textit{provided a single such model exists}, one can choose a geometrically reductive smooth affine model whose set of $A$-points contains any given bounded subgroup of $G(K)$. 

To deal with issues of disconnectedness, we will perform a pushout construction with certain finite flat (not necessarily \'etale) $A$-group schemes. To this end, we first develop a small amount of theory for extensions of such group schemes. In fact, we will develop a bit more than is necessary for the proof of Theorem~\ref{theorem:main-theorem}.

\begin{lemma}\label{lemma:extending-auts}
    Let $A$ be a DVR with fraction field $K$, and let
    \[
    1 \to M \to E_i \to \Gamma \to 1
    \]
    ($i = 1, 2$) be short exact sequences of finite flat $A$-group schemes such that $M$ is of multiplicative type and $\Gamma$ is \'etale. If $f_1: (E_1)_K \to (E_2)_K$ is a $K$-homomorphism preserving $M_K$, then $f_1$ extends uniquely to an $A$-homomorphism $f: E_1 \to E_2$. If $f_1$ is an isomorphism, then $f$ is also an isomorphism.
\end{lemma}

\begin{proof}
    If $f_1$ extends to an $A$-morphism $f: E_1 \to E_2$, then by flatness and schematic density of $\Spec K$ in $\Spec A$, this extension is unique and it is a homomorphism. Thus by descent, we may and do extend $A$ to assume that there are scheme-theoretic sections $s_i: \Gamma \to E_i$. From this we obtain isomorphisms of $A$-schemes $\vp_i: M \times_A \Gamma \to H$ given functorially by $\vp(m, \gamma) = m s_i(\gamma)$. Using the $\vp_i$, we see that $f_1$ induces a $K$-morphism $g_1: M_K \times_K \Gamma_K \to M_K \times_K \Gamma_K$ given functorially by $g_1(m, \gamma) = (f_1(m)d(\gamma), \overline{f}_1(\gamma))$, where $d: \Gamma_K \to M_K$ is a $K$-morphism. Thus it suffices to extend $f_1|_{M_K}$, $d$, and $\overline{f}_1$ over $A$. Since $M$ is finite and $\Gamma$ is \'etale, the fact that $d$ and $\overline{f}_1$ extend comes from the valuative criterion of properness. The fact that $f_1|_{M_K}$ extends comes from Cartier duality, noting that the Cartier dual of $M$ is finite \'etale.

    Now suppose that $f_1$ is an isomorphism with inverse $g_1$, and let $f$ and $g$ be the unique extensions. By uniqueness of extensions, we see that $f \circ g$ and $g \circ f$ are the respective identities.
\end{proof}

\begin{lemma}\label{lemma:extending-finite-group-schemes}
    Let $A$ be a DVR with fraction field $K$, and let $H$ be a finite $K$-group scheme. Suppose that we are given finite flat $A$-group schemes $H_0$ and $H_1$ and a short exact sequence
    \begin{align}\label{align:ses-1}
    1 \to (H_0)_K \to H \to (H_1)_K \to 1.
    \end{align}
    Suppose moreover that $H_0$ is of multiplicative type and $H_1$ is \'etale.
    
    Then there is a finite flat $A$-group scheme $\sH$ and a short exact sequence
    \begin{align}\label{align:ses-2}
    1 \to H_0 \to \sH \to H_1 \to 1
    \end{align}
    whose base change to $K$ is isomorphic to (\ref{align:ses-1}).
\end{lemma}

\begin{proof}
    Note that by Lemma~\ref{lemma:extending-auts}, the pair of the extension (\ref{align:ses-2}) and its isomorphism with (\ref{align:ses-1}) is unique up to unique isomorphism if it exists. Thus by descent, we may pass to a quasi-finite extension of $A$ to assume that $H_1$ is constant there is a scheme-theoretic section $s\colon (H_1)_K \to H$. Since $H_1$ is constant, Lemma~\ref{lemma:extending-auts} also shows that the induced conjugation map $(H_1)_K \times_K (H_0)_K \to (H_0)_K$, $(h_1, h_0) \mapsto s(h_1) h_0 s(h_1)^{-1}$ extends uniquely to an action map $\alpha\colon H_1 \times_A H_0 \to H_0$.
    
    Define $\sH$ scheme-theoretically as the product $H_0 \times H_1$. Note that $\sH_K \cong H$ via the map $f(h_0, h_1) = h_0 s(h_1)$. Under this isomorphism, the multiplication map on $H$ corresponds to the map $\mu_K: \sH_K \times_K \sH_K \to \sH_K$ given by
    \[
    \mu_K((h_0, h_1), (h_0', h_1')) = (h_0 \cdot s(h_1)h_0's(h_1)^{-1} \cdot c(h_1, h_1'), h_1 h_1')
    \]
    where $c(h_1, h_1') = s(h_1)s(h_1')s(h_1 h_1')^{-1}$. Since $H_1$ is \'etale and $H_0$ is finite, $c$ extends to an $A$-morphism $H_1 \times_A H_1 \to H_0$ which we will also denote by $c$. Similarly, $s(1) \in H_0(A).$
    
    Define an $A$-morphism $\mu: \sH \times_A \sH \to \sH$ extending $\mu_K$ by
    \[
    \mu((h_0, h_1), (h_0', h_1')) = (h_0 \cdot \alpha(h_1, h_0') \cdot c(h_1, h_1'), h_1 h_1')
    \]
    Diagrams of flat $A$-schemes may be checked to be commutative after passage to $K$, so it follows that $\mu$ is a monoid law (with identity $(s(1)^{-1}, 1)$). Moreover, $\mu$ is actually a group law: to see this functorially, let $B$ be an $A$-algebra and let $(h_0, h_1) \in \sH(B)$. By assumption on $\alpha$, there is a unique $h_0' \in H_0(B)$ such that $\alpha(h_1, h_0') = h_0^{-1} s(1)^{-1} c(h_1, h_1^{-1})^{-1}$. The right inverse of $(h_0, h_1)$ under $\mu$ is clearly $(h_0', h_1^{-1})$, and by the Yoneda lemma this gives a right inverse morphism $r: \sH \to \sH$. Note that $r$ is also a left inverse map (as one can check over $K$), so indeed $\sH$ is an $A$-group scheme.
    
    Finally, there are $A$-homomorphisms $i: H_0 \to \sH$ and $\pi: \sH \to H_1$ given by $i(h_0) = (h_0s(1)^{-1}, 1)$ and $\pi(h_0, h_1) = h_1$, and it is straightforward to check that these form a short exact sequence whose base change to $K$ is isomorphic to (\ref{align:ses-1}).
\end{proof}

The following lemma is a mild variant of \cite[Lemma 2.23]{booher-tang}.

\begin{lemma}\label{lemma:booher-tang-variant}
    Let $S$ be a connected scheme, and let 
    \begin{align}\label{align:extension-1}
    1 \to M \to E \to \Gamma \to 1
    \end{align}
    be a short exact sequence of finitely presented $S$-group schemes such that $M$ is of multiplicative type, $\Gamma$ is constant of order $n$, and $E \to \Gamma$ admits a scheme-theoretic section. There exists a short exact sequence
    \begin{align}\label{align:extension-2}
    1 \to M[n] \to H \to \Gamma \to 1
    \end{align}
    such that (\ref{align:extension-1}) is obtained from (\ref{align:extension-2}) by pushing forward along the inclusion $M[n] \to M$. Moreover, the pushout of (\ref{align:extension-2}) by the inclusion $M[n] \to M[n^2]$ is unique up to isomorphism.
\end{lemma}

\begin{proof}
    Since $E \to \Gamma$ admits a section, (\ref{align:extension-1}) corresponds to an element $\alpha$ of the Hochschild cohomology group $\mathrm{H}^2(\Gamma, M)$ (see \cite[Proposition 2.3.6]{Demarche}). Since $\Gamma$ is constant and $S$ is connected, $\mathrm{H}^i(\Gamma, M)$ agrees with the ordinary group cohomology $\mathrm{H}^i(\Gamma(S), M(S))$ for all $i \geq 0$. This latter group is killed by $n$ by classical theory, so $\alpha$ is the image of a class in $\mathrm{H}^2(\Gamma(S), M[n](S)) = \mathrm{H}^2(\Gamma, M[n])$; this corresponds to the desired extension (\ref{align:extension-2}). To see the final claim, consider the commutative diagram with exact rows
    \[
    \begin{tikzcd}
        \mathrm{H}^1(\Gamma, M) \arrow[r] \arrow[d, "n"]
        &\mathrm{H}^2(\Gamma, M[n]) \arrow[r] \arrow[d]
        &\mathrm{H}^2(\Gamma, M) \arrow[d] \\
        \mathrm{H}^1(\Gamma, M) \arrow[r]
        &\mathrm{H}^2(\Gamma, M[n^2]) \arrow[r]
        &\mathrm{H}^2(\Gamma, M)
    \end{tikzcd}
    \]
    Since $\mathrm{H}^1(\Gamma, M)$ is killed by $n$ (as above), a diagram chase concludes the proof.
\end{proof}

The following lemma is analogous to the fact that a (real) Lie group $G$ with finitely many connected components admits a maximal compact subgroup which meets every component of $G$.

\begin{lemma}\label{lemma:extending-reductive-groups}
    Let $A$ be a DVR with fraction field $K$, and let $G$ be a (possibly disconnected) reductive $K$-group. There exists a quasi-finite local extension of DVRs $A \subset A'$ inducing the extension of fraction fields $K \subset K'$ and a geometrically reductive smooth affine $A'$-integral model $\sG'$ of $G_{K'}$.
\end{lemma}

\begin{proof}
    By extending $A$, we may assume $G^0$ is split and $G/G^0$ is constant. First let $\sG_0$ be a split reductive $A$-group scheme with generic fiber $G^0$, which exists by the Existence and Isomorphism Theorems \cite[Exp.\ XXV, Th\'eor\`eme 1.1]{SGA3}. Let $(\sB_0, \sT_0)$ be a Borel-torus pair of $\sG_0$ over $A$, and let $N = N_G((\sB_0)_K, (\sT_0)_K)$, the normalizer of the pair $((\sB_0)_K, (\sT_0)_K)$ in $G$. By Lemma~\ref{lemma:booher-tang-variant}, if the constant group scheme $N/(\sT_0)_K \cong G/G^0$ is of order $n$, then after extending $K$ there is an extension
    \begin{align}\label{align:extension-3}
    1 \to (\sT_0)_K[n] \to H \to N/(\sT_0)_K \to 1
    \end{align}
    whose pushout along $(\sT_0)_K[n] \subset (\sT_0)_K$ is the tautological extension for $(\sT_0)_K \subset N$.
    If $\Gamma$ is the constant $A$-group scheme corresponding to $N/(\sT_0)_K \cong G/G^0$, then using Lemma~\ref{lemma:extending-finite-group-schemes} we find that there exists a finite flat $A$-group scheme $\sH$ and a short exact sequence
    \[
    1 \to \sT_0[n] \to \sH \to \Gamma \to 1
    \]
    whose generic fiber is (\ref{align:extension-3}).
    
    Now note that $G$ is isomorphic to the pushout $G^0 \times^{(\sT_0)_K[n]} H \coloneqq (G^0 \times_K H)/(\sT_0)_K[n]$. Thus if we define $\sG = \sG_0 \times^{\sT_0[n]} \sH$, we find that $\sG_K \cong G$. Moreover, $\sG^0 = \sG_0$ is reductive and $\sG/\sG^0 \cong \Gamma$ is finite \'etale, whence $\sG$ is a geometrically reductive smooth affine model of $G$.
\end{proof}

\section{Extending homomorphisms}\label{section:extending-homs}

In this section we aim to establish, essentially, that the quotient stack $[\underline{\Hom}_{S\textrm{-gp}}(H, G)/G]$ satisfies the existence part of the valuative criterion of properness when $S$ is a locally noetherian scheme, $H$ and $G$ are smooth $S$-affine $S$-group schemes, and $G$ is geometrically reductive. (This stack is often algebraic but it is usually not proper; it is locally of finite type and quasi-separated, but rarely quasi-compact or separated. See \cite[Section 2.1]{booher-tang} or \cite{Hom-schemes} for a discussion of $\underline{\Hom}_{S\textrm{-gp}}(H, G)$.) We begin with a simple lemma which mildly generalizes \cite[Corollary 2.10.10]{Kaletha-Prasad}.

\begin{lemma}\label{lemma:kaletha-prasad}
    Let $A \subset A'$ be a finite local extension of strictly henselian DVRs inducing the extension $K \subset K'$ of fraction fields. Let $X$ be a smooth $A$-scheme, and let $X'$ be an affine $A'$-scheme. A $K'$-morphism $f: X_{K'} \to X'_{K'}$ extends (uniquely) to an $A'$-morphism $X_{A'} \to X'$ if and only if $f(X(A)) \subset X'(A')$. In particular, if $f(X(A)) \subset X'(A')$ then $f(X(A')) \subset X'(A')$.
\end{lemma}

\begin{proof}
    By passing to an open cover, we may and do assume that $X$ is affine. If $f$ extends to an $A'$-morphism $X_{A'} \to X'$, then clearly $f(X(A)) \subset X'(A')$, so we need only show the converse. Note that $A \subset A'$ is flat, so the Weil restriction $\mathrm{R}_{A'/A}(X')$ exists as an affine $A$-scheme. The $K'$-morphism $f: X_{K'} \to X'_{K'}$ corresponds to a $K$-morphism $f': X \to \mathrm{R}_{K'/K}(X'_{K'})$ satisfying $f'(X(A)) \subset \mathrm{R}_{A'/A}(X')(A)$. By \cite[Corollary 2.10.10]{Kaletha-Prasad}, the $K$-morphism $f'$ extends to an $A$-morphism $X \to \mathrm{R}_{A'/A}(X')$, which in turn corresponds to an $A'$-morphism $X_{A'} \to X'$ extending $f$.
\end{proof}

Recall that, if $K$ is a discretely valued field with valuation $v$ and $X$ is an affine $K$-scheme, then a subset $B \subset X(K)$ is \textit{bounded} if, for every $f \in \Gamma(X, \sO_X)$, the function $b \mapsto v(f(b))$ is bounded below on $B$. For a detailed discussion of boundedness, see \cite[Section 2.2]{Kaletha-Prasad}. We need also standard properties of the (reduced) Bruhat--Tits building, which are summarized in \cite[Chapter 4]{Kaletha-Prasad}.

\begin{lemma}\label{lemma:parahoric-existence}
    Let $A$ be a henselian DVR with fraction field $K$, and let $G$ be a (possibly disconnected) reductive $K$-group. If $B \subset G(K)$ is a bounded subgroup, then there is a facet $F$ of the Bruhat--Tits building $\sB(G^0)$ such that $B$ stabilizes the barycenter of $F$.
\end{lemma}

\begin{proof}
    Note first that since $G(K)$ acts on $G^0(K)$ by conjugation, it also acts on $\sB(G^0)$, and the restriction of this action to $G^0(K)$ is the usual action. By \cite[Corollary 4.2.14]{Kaletha-Prasad}, there is a point $x_0$ of the (restricted) Bruhat--Tits building $\sB(G^0)$ such that $B \cap G^0(K)$ stabilizes $x_0$. Since $B \cap G^0(K)$ is of finite index in $B$, it follows that the set $Bx_0$ is finite. If $X$ is the convex hull of $Bx_0$ in $\sB(G^0)$, then $X$ is closed and bounded by \cite[Lemma 1.1.13]{Kaletha-Prasad}, and \cite[Theorem 4.2.12]{Kaletha-Prasad} shows that $X$ has a unique barycenter $x_1$, invariant under all isometries of $\sB(G^0)$ which preserve $X$. In particular, the action of $B$ on $\sB(G^0)$ preserves $x_1$. If $F$ is the (open) facet of $\sB(G^0)$ containing $x_1$, then the action of $B$ preserves the barycenter of $F$, as desired.
\end{proof}

If $G$ is a geometrically reductive smooth affine group scheme over a henselian DVR $A$ with fraction field $K$ and valuation $v$, then we let $G(K)^1$ be the open subgroup $G^0(K)^1 \cdot G(A)$ of $G(K)$, where $G^0(K)^1$ is defined as in \cite[Section 2.6(d)]{Kaletha-Prasad}:
\[
G^0(K)^1 \coloneqq \{g \in G^0(K): v(\chi(g)) = 0 \text{ for all } \chi \in X^*(G^0)\},
\]
where $X^*(G^0)$ is the group of characters $\chi: G^0 \to \bG_m$. As in the connected case, every bounded subgroup of $G(K)$ is contained in $G(K)^1$ because $G(K)/G(K)^1$ is topologically isomorphic to a subgroup of the torsion-free discrete group $X^*(T)$, where $T$ is the maximal central $K$-torus of $G^0$. Notice that $G^0(K)^1 = G(K)^1 \cap G^0(K)$.

\begin{lemma}\label{lemma:extending-reductive-groups-2}
    Let $A$ be an excellent henselian DVR with fraction field $K$, and let $G$ be a geometrically reductive smooth affine $A$-group scheme. If $B \subset G(K)$ is a bounded subgroup, then there exists a finite local extension of DVRs $A \subset A'$ with fraction field $K'$ and a geometrically reductive smooth affine $A'$-integral model $G'$ of $G_{K'}$ such that $B \subset G'(A')$. Moreover, for any such $G'$, there is a further extension of $A'$ such that the subgroup $G'(A')$ is $G^0(K')$-conjugate to $G(A')$.
\end{lemma}

\begin{proof}
    By spreading out, we may and do assume that $A$ is strictly henselian. By Lemma~\ref{lemma:parahoric-existence}, there is a facet $F$ of the Bruhat--Tits building $\sB(G^0)$ such that $B$ stabilizes the barycenter $x$ of $F$. Since $x$ is the barycenter of a facet, the argument of \cite[Lemma 2.4]{Larsen-maximality} shows that there exists a finite extension $K'/K$ with valuation ring $A'$ such that the image $x'$ of $x$ in the building $\sB(G^0_{K'})$ is hyperspecial.\footnote{The proof of \cite[Lemma 2.4]{Larsen-maximality} appears to have a small gap, which can be fixed by noting that one may first pass to a finite extension $K'/K$ such that $x'$ is a vertex. Indeed, if $\sA$ is an apartment containing $x$ and corresponding to a maximal split $K$-torus $T$, then $\sA$ is an affine space under $V(T) \coloneqq \bR \otimes_{\bZ} X_*(T)$, and the metric on $\sB(G^0)$ restricts to the Euclidean metric on $\sA$ (see \cite[Section 4.2]{Kaletha-Prasad}). After choosing an identification of $\sA$ and $V(T)$, one sees that $x$ is a $\bQ$-linear combination of vertices, and the claim follows from the discussion in \cite[Section 6.5]{Kaletha-Prasad}. The phrase ``As $x_1$ is a facet" from \cite{Larsen-maximality} should then be replaced with ``As $x_1$ is a vertex".} Let $U'$ be the stabilizer of $x'$ in $G(K')^1$, so $B \subset U'$. By \cite[Propositions 8.3.1, A.7.1]{Kaletha-Prasad}, there is a smooth affine model $G'$ of $G_{K'}$ over $A'$ such that $G'(A') = U'$.
    
    Since $G^0$ is a reductive group scheme, \cite[Theorem 9.9.3(2)]{Kaletha-Prasad} shows that $G^0(A')$ stabilizes a hyperspecial point $y$ of $\sB(G^0_{K'})$. By \cite[Proposition 1.3.43(3), Corollary 7.4.8]{Kaletha-Prasad}, the point $y$ is the \textit{unique} one stabilized by $G^0(A')$. Moreover, Lemma~\ref{lemma:parahoric-existence} shows that $G(A')$ stabilizes some point of $\sB(G^0)$, so in fact $G(A')$ stabilizes $y$. By \cite[Proposition 10.2.2]{Kaletha-Prasad}, we may pass to a finite extension of $K'$ to find some $g \in G^0(K')$ such that $g\cdot x = y$, and thus $gG'(A')g^{-1} = G(A')$. By Lemma~\ref{lemma:kaletha-prasad}, the $K'$-isomorphism $G_{K'} \to G_{K'}$ given by $g$-conjugation induces an $A'$-isomorphism $G' \to G$. Because $A$ is excellent henselian, the extension $A \subset A'$ is finite.
\end{proof}

\begin{lemma}\label{lemma:griess-ryba}
    Let $A$ be an excellent henselian DVR with fraction field $K$, and let $G$ and $H$ be smooth affine $A$-group schemes such that $G$ is geometrically reductive. If $f_1: H_K \to G_K$ is a $K$-homomorphism, then there exists a finite local extension of DVRs $A \subset A'$ with fraction field $K'$ and $g \in G(K')$ such that $\Ad(g) \circ f_1$ extends to an $A'$-homomorphism $f': H_{A'} \to G_{A'}$.
\end{lemma}

\begin{proof}
    By spreading out, we may and do assume that $A$ is strictly henselian with algebraically closed residue field. Note that $f_1(H(A)) \subset G(K)$ is a bounded subgroup, so Lemma~\ref{lemma:extending-reductive-groups-2} shows that there is a finite local extension $A \subset A'$, a geometrically reductive smooth affine $A'$-group scheme $G'$ such that $f_1(H(A)) \subset G'(A')$, and $g \in G^0(K')$ such that $gG'(A')g^{-1} = G(A')$. By Lemma~\ref{lemma:kaletha-prasad}, it follows that $\Ad(g) \circ (f_1)_{K'}: H_{K'} \to G'_{K'}$ extends to an $A'$-homomorphism $f': H_{A'} \to G'$, as desired.
\end{proof}

We need one more technical lemma before proving Theorem~\ref{theorem:main-theorem}.

\begin{lemma}\label{lemma:technical}
    Let $K$ be a discretely valued field, let $G$ be a reductive $K$-group, let $\Gamma$ be a finitely generated group, and let $\vp, \psi: \Gamma \to G(K)$ be two homomorphisms whose $G$-orbits have intersecting closures in $\underline{\Hom}_{K\textrm{-gp}}(\Gamma, G)$. Suppose moreover that $\vp(\Gamma)$ is bounded in $G(K)$ and that the Zariski closure $G_1 = \overline{\psi(\Gamma)}$ is reductive. Then $\psi(\Gamma)$ is bounded in $G(K)$.
\end{lemma}

\begin{proof}
    Note first that if $\Gamma' \subset \Gamma$ is a finite index subgroup, then $\psi(\Gamma)$ is bounded if and only $\psi(\Gamma')$ is bounded; thus we may shrink $\Gamma$ to assume that $G_1$ is connected. Let $\rho: G \to \GL(V)$ be a faithful $K$-representation, and note that boundedness of $\vp(\Gamma)$ implies that for each $\gamma \in \Gamma$, all of the eigenvalues of $\rho(\vp(\gamma))$ on $V_{\overline{K}}$ are integral. Because the $G$-orbits of $\vp$ and $\psi$ have intersecting closures, the eigenvalues of $\rho(\psi(\gamma))$ on $V_{\overline{K}}$ are the same as those of $\rho(\vp(\gamma))$, and in particular they are all integral. Thus by \cite[Lemma 2.2.11]{Kaletha-Prasad} (applied to the connected reductive group $G_1$), we find that $\psi(\Gamma)$ is bounded, as desired.
\end{proof}

\section{Proof of Theorem~\ref{theorem:main-theorem}}\label{section:main-proof}

As in the introduction, let $S$ be a locally noetherian scheme, and let $G$ and $H$ be geometrically reductive smooth affine $S$-group schemes. Fix a finite $S$-homomorphism $f: H \to G$ and a homomorphism $h: \Gamma' \to \Gamma$ of finitely generated groups such that $h(\Gamma')$ is of finite index in $\Gamma$. For simplicity, write $\bH_H = \underline{\Hom}_{S\textrm{-gp}}(\Gamma, H)$ and $\bH'_G = \underline{\Hom}_{S\textrm{-gp}}(\Gamma', G)$. We will show in this section that the $S$-morphism
\[
    F: \bH_H/\!/H \to \bH'_G/\!/G
\]
is finite. \smallskip

The map $F$ is evidently affine, and it is of finite type by \cite[Theorem 6.3.3]{Alper-adequate}, so to show that $F$ is finite it suffices (by \cite[\href{https://stacks.math.columbia.edu/tag/01WM}{Tag 01WM}]{stacks-project}) to show that it is universally closed, or equivalently to verify the existence part of the valuative criterion of properness for $F$. In other words, we must show that if $A$ is a DVR with fraction field $K$ and we have a solid diagram
\[
\begin{tikzcd}
    \Spec K' \arrow[r, dotted] \arrow[d, dotted]
    &\Spec K \arrow[r, "y"] \arrow[d]
    &\bH_H/\!/H \arrow[d, "F"] \\
    \Spec A' \arrow[r, dotted] \arrow[rru, dotted]
    &\Spec A \arrow[r, "x"]
    &\bH'_G/\!/G
\end{tikzcd}
\]
then there is a local extension $A \subset A'$ of DVRs inducing the extension $K \subset K'$ of fraction fields such that the above dashed diagram can be filled in to become commutative. We will extend $A$ in stages; note that at every stage we are free to replace $A$ by its completion to assume that $A$ is excellent and henselian.\smallskip

From now on, fix a solid diagram as above. By \cite[Theorem 5.3.1(5)]{Alper-adequate}, every element of $(\bH_H/\!/H)(\overline{K})$ lifts to a homomorphism $\psi: \Gamma \to H(\overline{K})$ with closed $H_{\overline{K}}$-orbit in $\bH_{H_{\overline{K}}}$. Since $\Gamma$ is finitely generated, we may pass to a local extension of $A$ to assume that $y$ lifts to such a homomorphism $\psi$ valued in $H(K)$ with closed $H$-orbit in $\bH_H$. Concretely, we must show that there is a local extension $A \subset A'$ of DVRs with fraction field $K'$ such that $\psi$ is $H(K')$-conjugate to a homomorphism $\Gamma \to H(A')$, or equivalently that $\psi(\Gamma)$ is $H(K')$-conjugate to a subgroup of $H(A')$. Let $H_1$ be the closed $K$-subgroup of $H_K$ which is topologically generated by $\psi(\Gamma)$. A surjection $F_N \to \Gamma$ from a free group $F_N$ induces an $H$-equivariant closed embedding $\bH_H \to H^N$. In particular, the image of $\psi$ under such an embedding still has closed $H$-orbit, so \cite[Lemma 16.3, Theorem 16.4]{Richardson} shows that $H_1$ is reductive. (Strictly speaking, this reference requires $H$ to be connected, but its proof works in the non-connected case because of the robustness of the results of \cite{Borel-Tits} on which it relies.) \smallskip

By Lemma~\ref{lemma:git-quotient-points}, we may extend $A$ to assume that $x$ is the image of a homomorphism $\vp: \Gamma' \to G(A)$ under the natural map $\bH'_G(A) \to (\bH'_G/\!/G)(A)$. Because $f \circ \psi \circ h$ and $\vp$ have the same image in $(\bH'_G/\!/G)(K)$, \cite[Theorem 5.3.1(5)]{Alper-adequate} shows that the closures of their orbits intersect. Moreover, if $G_1$ is the Zariski closure of $f(\psi(h(\Gamma')))$, then there is a finite surjective map $H_1^0 \to G_1^0$, so $G_1$ is reductive because $H_1$ is reductive. By Lemma~\ref{lemma:technical}, it follows that $f(\psi(h(\Gamma')))$ is bounded. Since $f(\psi(\Gamma))$ contains $f(\psi(h(\Gamma')))$ as a finite index subgroup, it follows that $f(\psi(\Gamma))$ is bounded in $G(K)$. By \cite[Lemma 2.2.10]{Kaletha-Prasad}, the fact that $f$ is finite implies that $\psi(\Gamma)$ is bounded in $H(K)$.\smallskip

By the previous paragraph, Lemma~\ref{lemma:extending-reductive-groups}, and Lemma~\ref{lemma:extending-reductive-groups-2}, after extending $A$ further we may and do assume that there is a geometrically reductive smooth affine $A$-group scheme $\sH_1$ with generic fiber $H_1$ such that $\psi(\Gamma) \subset \sH_1(A)$. By construction, there is a $K$-morphism $i_1: H_1 \to H_K$, and by Lemma~\ref{lemma:griess-ryba} we may pass to a local extension of $A$ and conjugate by an element of $H(K)$ to assume that $i_1$ extends to an $A$-homomorphism $i: \sH_1 \to H$. Thus $\psi(\Gamma)$ is $H(K)$-conjugate to a subgroup of $H(A)$, verifying the existence part of the valuative criterion of properness and proving that $F$ is finite, as desired.

\begin{remark}
    If $S$ is not required to be locally noetherian, then it follows from Theorem~\ref{theorem:main-theorem}, \cite[Proposition 5.2.9]{Alper-adequate}, and spreading out that $F$ is integral. However, I do not know whether the schemes involved are finitely presented (or even of finite type) in this case, so \textit{finiteness} of the morphism is not clear.
\end{remark}

\bibliography{bibliography}

\begin{thebibliography}{DHKM22}
\expandafter\ifx\csname url\endcsname\relax
  \def\url#1{\texttt{#1}}\fi
\expandafter\ifx\csname doi\endcsname\relax
  \def\doi#1{\burlalt{doi:#1}{http://dx.doi.org/#1}}\fi
\expandafter\ifx\csname urlprefix\endcsname\relax\def\urlprefix{URL }\fi
\expandafter\ifx\csname href\endcsname\relax
  \def\href#1#2{#2}\fi
\expandafter\ifx\csname burlalt\endcsname\relax
  \def\burlalt#1#2{\href{#2}{#1}}\fi

\bibitem[Alp14]{Alper-adequate}
J.~Alper.
\newblock Adequate moduli spaces and geometrically reductive group schemes.
\newblock {\em Algebr. Geom.}, 1(4):489--531, 2014.
\newblock \doi{10.14231/AG-2014-022}.

\bibitem[BCT23]{booher-tang}
J.~Booher, S.~Cotner, and S.~Tang.
\newblock Lifting {$G$}-{Valued} {Galois} {Representations} when $\ell \neq p$,
  2023, \burlalt{2211.03768}{http://arxiv.org/abs/2211.03768}.

\bibitem[BT71]{Borel-Tits}
A.~Borel and J.~Tits.
\newblock \'{E}l\'{e}ments unipotents et sous-groupes paraboliques de groupes
  r\'{e}ductifs. {I}.
\newblock {\em Invent. Math.}, 12:95--104, 1971.
\newblock \doi{10.1007/BF01404653}.

\bibitem[Con14]{Conrad}
B.~Conrad.
\newblock Reductive group schemes.
\newblock In {\em Autour des sch\'{e}mas en groupes. {V}ol. {I}}, volume 42/43
  of {\em Panor. Synth\`eses}, pages 93--444. Soc. Math. France, Paris, 2014.

\bibitem[Cot23]{Hom-schemes}
S.~Cotner.
\newblock Hom schemes for algebraic groups, 2023,
  \burlalt{2309.16458}{http://arxiv.org/abs/2309.16458}.

\bibitem[Dat22]{Dat-IHES}
J.-F. Dat.
\newblock Moduli spaces of local {L}anglands parameters.
\newblock
  \url{https://webusers.imj-prg.fr/~jean-francois.dat/recherche/publis/ihes.pdf},
  2022.

\bibitem[Dem15]{Demarche}
C.~Demarche.
\newblock Cohomologie de {H}ochschild non {A}b\'{e}lienne et extensions de
  {F}aisceaux en groupes.
\newblock In {\em Autours des sch\'{e}mas en groupes. {V}ol. {II}}, volume~46
  of {\em Panor. Synth\`eses}, pages 255--292. Soc. Math. France, Paris, 2015.

\bibitem[DHKM20]{dhkm_moduli}
J.-F. Dat, D.~Helm, R.~Kurinczuk, and G.~Moss.
\newblock Moduli of {Langlands} {Parameters}, 2020,
  \burlalt{2009.06708}{http://arxiv.org/abs/2009.06708}.

\bibitem[DHKM22]{dhkm_finiteness}
J.-F. Dat, D.~Helm, R.~Kurinczuk, and G.~Moss.
\newblock Finiteness for {Hecke} algebras of $p$-adic groups, 2022,
  \burlalt{2203.04929}{http://arxiv.org/abs/2203.04929}.

\bibitem[GR98]{Griess-Ryba}
R.~L. Griess, Jr. and A.~J.~E. Ryba.
\newblock Embeddings of {${\mathrm PGL}_2(31)$} and {${\mathrm SL}_2(32)$} in
  {$E_8(\bold C)$}.
\newblock {\em Duke Math. J.}, 94(1):181--211, 1998.
\newblock \doi{10.1215/S0012-7094-98-09409-1}.
\newblock With appendices by Michael Larsen and J.-P. Serre.

\bibitem[KP23]{Kaletha-Prasad}
T.~Kaletha and G.~Prasad.
\newblock {\em Bruhat-{T}its theory---a new approach}, volume~44 of {\em New
  Mathematical Monographs}.
\newblock Cambridge University Press, Cambridge, 2023.

\bibitem[Lar95]{Larsen-maximality}
M.~Larsen.
\newblock Maximality of {G}alois actions for compatible systems.
\newblock {\em Duke Math. J.}, 80(3):601--630, 1995.
\newblock \doi{10.1215/S0012-7094-95-08021-1}.

\bibitem[Mar03]{martin03}
B.~M.~S. Martin.
\newblock Reductive subgroups of reductive groups in nonzero characteristic.
\newblock {\em J. Algebra}, 262(2):265--286, 2003.
\newblock \doi{10.1016/S0021-8693(03)00189-3}.

\bibitem[Ric88]{Richardson}
R.~W. Richardson.
\newblock Conjugacy classes of {$n$}-tuples in {L}ie algebras and algebraic
  groups.
\newblock {\em Duke Math. J.}, 57(1):1--35, 1988.
\newblock \doi{10.1215/S0012-7094-88-05701-8}.

\bibitem[SGA3]{SGA3}
P.~Gille and P.~Polo, editors.
\newblock {\em Sch\'{e}mas en groupes ({SGA} 3).}
\newblock 2011.
\newblock S\'{e}minaire de G\'{e}om\'{e}trie Alg\'{e}brique du Bois Marie
  1962--64. [Algebraic Geometry Seminar of Bois Marie 1962--64], A seminar
  directed by M. Demazure and A. Grothendieck with the collaboration of M.
  Artin, J.-E. Bertin, P. Gabriel, M. Raynaud and J-P. Serre, Revised and
  annotated edition of the 1970 French original.

\bibitem[{Sta}23]{stacks-project}
T.~{Stacks project authors}.
\newblock The {S}tacks project.
\newblock \url{https://stacks.math.columbia.edu}, 2023.

\bibitem[Vin96]{Vinberg}
E.~B. Vinberg.
\newblock On invariants of a set of matrices.
\newblock {\em J. Lie Theory}, 6(2):249--269, 1996.

\end{thebibliography}

\end{document}